\documentclass[11pt]{article}
\usepackage[utf8]{inputenc}
\usepackage{fullpage}
\usepackage{amsmath}
\usepackage{amsthm}
\usepackage{amscd,amssymb}
\usepackage[all]{xy}
\usepackage{color}
\usepackage{comment}
\usepackage{appendix}
\usepackage{verbatim}
\usepackage{pdfsync}
\usepackage{hyperref}
\usepackage{cleveref}
\usepackage{float}
\usepackage{setspace}
\usepackage{xpatch}
\usepackage{subcaption}
\usepackage{pst-node}
\usepackage{tikz-cd}
\usetikzlibrary{positioning}
\usetikzlibrary{shapes,arrows}
\usepackage{tikz}
\usepackage{graphicx,epsfig}
\usepackage{mathrsfs}
\usepackage{verbatim}
\date{}

\newtheorem{theorem}{Theorem}
\newtheorem{lemma}[theorem]{Lemma}
\newtheorem{cor}[theorem]{Corollary}
\newtheorem{proposition}[theorem]{Proposition}
\theoremstyle{definition}
\newtheorem{remark}[]{Remark}
\newtheorem{definition}{Definition}
\newtheorem{claim}[]{Claim}

\newtheorem{example}[]{Example}

\newcommand{\bS}{\mathbb{S}}
\newcommand{\supp}{\mathrm{Supp}}
\def\ind{\mathrm{Ind}}
\def\st{\mathrm{st}}
\def\bd{\mathrm{Bd}}
\def\S{\mathcal{S}}

\title{Distance $r$-domination number and $r$-independence complexes of graphs}
\author{Priyavrat Deshpande\footnote{Chennai Mathematical Institute, Chennai, India. Email: pdeshpande@cmi.ac.in}, Samir Shukla\footnote{Indian Institute of Technology Bombay, Mumbai, India. Email: samirshukla43@gmail.com}, Anurag Singh\footnote{Chennai Mathematical Institute, Chennai, India. Email: anuragsingh@cmi.ac.in}}
\date{}

\begin{document}
	
	\maketitle
	
	\begin{abstract}
		For $r\geq 1$, the $r$-independence complex of a graph $G$, denoted $\ind_r(G)$, is a simplicial complex whose faces are subsets $A \subseteq V(G)$ such that each component of the induced subgraph $G[A]$ has at most $r$ vertices. In this article, we establish a relation between the distance $r$-domination number of $G$ and (homological) connectivity of $\ind_r(G)$. We also prove that $\ind_r(G)$, for a chordal graph $G$, is either contractible or homotopy equivalent to a wedge of spheres. Given a wedge of spheres, we also provide a construction of a chordal graph whose $r$-independence complex has the homotopy type of the given wedge.
	\end{abstract}
	


\vspace{.1in}

\hrule
	
	\section{Introduction}

	The {\it independence complex}, $\ind(G)$, of a graph $G$ is the simplicial complex whose simplices are those subsets $I$ of vertices of $G$ such that the induced subgraph $G[I]$ does not have any edge. Independence complexes have applications in  several areas. Study of topological properties of independence complexes has been an active direction of research. For example,
	Babson and Kozlov \cite{BK1} used the topology of independence complexes of cycles to prove a conjecture by Lov\'{a}sz. Properties of independence complexes have also been used to study the Tverberg graphs \cite{ea} and 
	the independent  system of representatives \cite{abz}.
	%Independence complexes have been studied for various classes of graphs, for example, triangle free graphs \cite{Bar13}, claw free graphs \cite{ea1},  chordal graphs \cite{kk} etc.
	For more on these complexes, interested reader  is referred to
	\cite{Bar13, BLN, BB, eh, ea1, kk}. 
	%Meshulam, in \cite{RM}, gave a connection  between the domination number of a graph $G$ and certain homological properties of $\ind(G)$; and their application to Hall-type theorems for coloured  independent sets.  

	Let $r$ be a positive integer and $G$ be a graph. A set $A\subseteq V(G)$ is called {\itshape  $r$-independent} if each connected component of the induced subgraph $G[A]$ has at most $r$ vertices. 
	In \cite{PS18}, Paolini and Salvetti generalized the concept of independence complex by defining $r$-independence complex of a graph $G$ for any $r \geq 1$. In the same paper, they gave a relation between twisted  homology of classical braid groups and the homology of $r$-independence  complexes of associated Coxeter graphs. 
	The {\itshape $r$-independence complex} of $G$, denoted $\ind_r(G)$, is the simplicial complex whose simplices are $r$-independent sets of vertices of $G$. Observe that $\ind_1(G)$ is same as $\ind(G)$. First and third authors \cite{DS19} initiated the study of these complexes and gave a closed form formula for the homotopy type of $r$-independence complexes for certain families of graphs including complete $s$-partite graphs, fully whiskered graphs, cycle graphs and perfect $m$-ary trees.
	
	Let $\Gamma(G)$ and $\Gamma_0(G)$ denote the domination number and strong domination number of a graph $G$ respectively.  For $k\geq 0$, let $\tilde{H}_k(X)$ denotes the $k^{\text{th}}$ reduced homology of a topological space $X$ with integer coefficients. In \cite{RM}, Meshulam proved the following.
	
	\begin{theorem}\cite[Theorem 1.2]{RM} \label{thm:meshulam}\begin{itemize}
			\item[(i)] If $\Gamma_0(G) > 2k$, then $\tilde{H}_{k-1}(\ind_1(G)) = 0$.
			\item[(ii)] If $G$ is chordal and $\Gamma(G) > k$, then  $\tilde{H}_{k-1}(\ind_1(G)) = 0$.
		\end{itemize}
	\end{theorem}
	
	It is natural to ask, whether we can relate some topological properties of $\ind_r(G)$ to some graph theoretic invariants of $G$. One of the main motivations of this article is to establish similar results for $r$-independence complexes.
	
	\begin{definition}
		A set $D \subseteq V(G)$ is said to be a {\itshape distance $r$-dominating set} of $G$ if the distance $d(u, D)$ between each vertex $u \in V(G)\setminus D$ and $D$ is at most $r$. The minimum cardinality of a distance $r$-dominating set in $G$ is the {\itshape distance $r$-domination number} of $G$, denoted by $\gamma_r(G)$. 
	\end{definition}
	
	The distance $r$-dominating number of graphs is a well studied  notion in graph theory. For more about this invariant see, for example \cite{davila, jerold, TX}. 
	
	\begin{definition} \label{def:dominate}
		Let $S_r$ be a collection of connected subgraphs of $G$ of cardinality at most $r$. The collection $S_r$  is called a {\itshape dominating $r$-collection} if for each $v \in V(G)$ there exists an element $S \in S_r$ such that $d(v, S)$ is at most $1$. Given $r \geq 1$, the {\itshape $r$-set domination number} of $G$, denoted by $\omega_r(G)$ is the minimum $m$ such that there exists dominating $r$-collection of cardinality 
		$m$, {\it i.e.}, 
		$$
		\omega_r(G) : = \min\{|S_r|  :   S_r \ \text{is a dominating} \ \text{$r$-collection}\}.
		$$
	%	If there is no dominating $r$-collection, then define $\omega_r(G):=\infty$.
	\end{definition}
	
	Clearly, $\gamma_1(G)=\omega_1(G) = \Gamma(G)$ and $\omega_r(G) \geq \gamma_r(G)$ for each $r\geq 1$. The following example shows that the gap between $\omega_r(G)$ and $\gamma_r(G)$ can be arbitrarily large for any $r\geq 2$.
	
	\begin{example}
		Let $G$ be the graph shown in \Cref{fig:long_star_graph}. Since $d(v_1,w)$ is at most $2$ for all vertices in $G$, $\gamma_2(G)=1$. Further, it is easy to see that $\omega_2(G)=5$. We can attach more paths of length $3$ with vertex $v_1$ to increase the number $w_2$ by keeping $\gamma_2$ constant. A similar construction can be done for any $r >2$.
	\end{example}
	
	\begin{figure}[H]
		\centering
		\begin{tikzpicture}
		[scale=0.25, vertices/.style={draw, fill=black, circle, inner sep=1.5pt}]
		\node[vertices, label=above:{$v_1$}] (v1) at (0,8)  {};
		\node[vertices] (l11) at (-5.0,4.9)  {};
		\node[vertices] (l12) at (-3.1,4.8)  {};
		\node[vertices] (l13) at (-0.95,4.2)  {};
		\node[vertices] (l14) at (1.0,4.4)  {};
		\node[vertices] (l1m) at (3.0,4.6)  {};
		\node[vertices] (l21) at (-6.5,.7)  {};
		\node[vertices] (l22) at (-4.1,.3)  {};
		\node[vertices] (l23) at (-1.5,0.2)  {};
		\node[vertices] (l24) at (2.4,0.4)  {};
		\node[vertices] (l2m) at (5.0,.7)  {};
		%\node[vertices,inner sep=0.3pt] (d1) at (-0.5,0.8)  {};
		%\node[vertices,inner sep=0.3pt] (d2) at (0.3,0.7)  {};
		%\node[vertices,inner sep=0.3pt] (d3) at (1.1,0.65)  {};
		
		\foreach \to/\from in {v1/l11,v1/l12,v1/l13,v1/l14,v1/l1m,l11/l21,l12/l22,l13/l23,l14/l24,l1m/l2m}
		\draw [-] (\to)--(\from);
		\end{tikzpicture}\caption{}\label{fig:long_star_graph}
	\end{figure}

	The main results of this article are following.
	\begin{theorem}[See \Cref{thm:maindominationinside}]\label{theorem:main_domination}
		Let $G$ be a graph and $r \geq 1$. 
		\begin{itemize}
			\item[(i)] If $\gamma_r(G) > 2k$, then $\tilde{H}_{j}(\ind_r(G)) = 0$ for all $j \leq k+r-2$. 
			\item[(ii)] If $\omega_r(G) > 2k$, then $\tilde{H}_{j}(\ind_r(G)) = 0$ for all $j \leq k-1$. 
			
		\end{itemize}
	\end{theorem}

	% We also prove that the higher independence complexes of chordal graphs are either contractible or homotopy equivalent to wedge of sphere. 
	%Further, we give an lower bounds for the dimension of spheres in terms of $\omega_r(G)$, when $G$ is chordal and $\ind_r(G)$ is homotopy equivalent to wedge of spheres.
	
	\begin{theorem}[See \Cref{theorem:lowerboundconnectivityinside}] \label{theorem:lowerboundconnectivity}  Let $G$ be a chordal graph and $r \geq 1$.
		\begin{itemize}
			\item[(i)] $\ind_r(G)$ is either contractible or homotopy equivalent to wedge of spheres.
			
			% \item[(ii)] If $\ind_r(G)\simeq  \bigvee \mathbb{S}^{i_k}$, then $i_k \geq r\omega_r(G)-1$ and for each $k$ there exists an $s_k \in \mathbb{N}$ such that $i_k = rs_k-1$.

			\item[(ii)] If  $\omega_r(G) > k$, then $\tilde{H}_{i}(\ind_r(G)) = 0$ for each $i \leq rk-1$.
			\item[(iii)] If $\ind_r(G)\simeq  \bigvee \mathbb{S}^{i_k}$, then for each $i_k$ there exists a positive integer $s_k $ such that $i_k = rs_k-1$.
		\end{itemize}
	\end{theorem}
	
	\begin{theorem} [See \Cref{thm:reverseconstructioninside}]\label{thm:reverseconstruction}
		Let $r \geq 2$. Let $(d_1, \ldots, d_n)$ and $(k_1, \ldots, k_n)$ be two sequences of positive integers. There exists a chordal graph $G$ such that $\ind_r(G) \simeq \bigvee\limits_{i=1}^n \vee_{d_i} \mathbb{S}^{rk_i-1}$. 
	\end{theorem}

	This article is organized as follows: In \Cref{sec:prel}, we give basic definitions and results which are used in the remaining sections. \Cref{sec:dominationnumber} is dedicated to the proof of \Cref{theorem:main_domination}. Proof of \Cref{theorem:lowerboundconnectivity} and \Cref{thm:reverseconstruction} is given in \Cref{sec:chordal}.
	
	\section{Preliminaries}\label{sec:prel}
	
	A {\it graph} is an ordered pair $G=(V(G),E(G))$ where $V(G)$ is called the set of vertices and $E(G) \subseteq V(G) \times V(G)$, the set of unordered edges of $G$.  If $G$ is a graph on $n$ vertices, then we also say that $G$ is of cardinality $n$. The vertices $v_1, v_2 \in V(G)$ are said to be adjacent, if $(v_1,v_2)\in E(G)$. This is also denoted by $v_1 \sim v_2.$  For a subset $U \subset V(G)$, the {\it induced subgraph} $G[U]$ is the subgraph whose set of vertices  $V(G[U]) = U$
	and the set of edges
	$E(G[U]) = \{(a, b) \in E(G) \ | \ a, b \in U\}$. We also denote the graph $G[V(G) \setminus U]$ by $G- U$. For a graph $G$ and $S \subseteq V(G)$, let $N(S) := \{ v \in V(G) : v \sim s \mathrm{ ~for ~some~} s \in S\}$ and $N[S] = N(S) \cup S$.  
	
	For two distinct vertices $u$ and $v$, the distance $d(u, v)$ between $u$ and $v$ is the length of a shortest path between $u$ and $v$. Here, the length of a path is the number of edges in that path. 	If $X$ and $Y$ are two disjoint subsets of $V(G)$, then the distance between $X$ and $Y$ is defined as $d(X,Y) = \text{min}\{d(x,y): x \in X,y\in Y\}$.

	A subset $S \subseteq V(G)$ is  called a {\it dominating set} if for each vertex $v \in V(G) \setminus S$, there exists a $s \in S$ such that $v \sim s$. The {\it domination number} $\Gamma(G)$ is the minimum cardinality of a dominating set. Set $S$ is called a {\it strong dominating set} if each vertex  $v \in V(G)$ is adjacent to some vertex of $S$. The {\it strong domination number} $\Gamma_0(G)$ is the minimum cardinality of a strong dominating set. 
	\begin{definition}
		For $r\geq 0$, $v \in V(G)$, $S \subseteq V(G)$ is called  an {\it $r$-support} of $v$ in $G$, if $v \notin S, |S|=r$ and $G[S\cup \{v\}]$ is a connected graph. Let Supp$_r(v,G)$ denote the collection of all $r$-supports in $G$, {\it i.e.}, $$\mathrm{Supp}_r(v,G) = \{S : S \text{ is an $r$-support of $v$ in $G$}\}.$$
		Clearly, $\supp_0(v,G)=\{\emptyset\}$. We say that Supp$_r(v,G)$ is {\itshape connected}, if $G[S]$ is connected for all $S\in \mathrm{Supp}_r(v,G).$ Whenever the underlying graph is clear, we simply denote it by Supp$_r(v)$. 
	\end{definition}
	
	%For $S\subseteq V(G)$, let $N[S]$ denote the set of all elements of $S$ and their neighbors, {\itshape i.e.}, $N[S]= S \cup \{w : w \in N(s) \text{ for some } s \in S\}.$

	A {\it finite abstract simplicial complex} $K$ is a collection of
	finite sets such that if $\tau \in K$ and $\sigma \subset \tau$,
	then $\sigma \in K$. The elements  of $K$ are called {\it simplices}
	of $K$. The  dimension of a simplex $\sigma$ is equal to $|\sigma| - 1$, here $|\cdot|$ denote the cardinality. 
	The dimension of an abstract  simplicial complex is the maximum of the dimensions of its simplices. The $0$-dimensional
	simplices are called vertices of $K$. If $\sigma \subset \tau$, we say that
	$\sigma$ is a face of $\tau$.  If a simplex has dimension $k$, it is said to
	be $k${\it -dimensional} or   $k$-{\it simplex}.  The {\it boundary} of a $k$-simplex
	$\sigma $ is the simplicial complex, consisting of all faces of $\sigma$
	of dimension $\leq k-1$ and it is denoted by $\bd(\sigma).$
	The {\it star} of a simplex $\sigma \in K$ is the subcomplex of $K$ defined as
	$$
	\st_K(\sigma):= \{\tau \in K \ |  \ \sigma \cup \tau \in K\}.
	$$
	
	Whenever the underlying space $K$ is clear, we write $\st(\sigma)$ to denote $\st_K(\sigma)$.
	
	\begin{definition} Let $K_1$ and $K_2$ be two simplicial complexes whose vertices are indexed by disjoint sets. The join of $K_1$ and $K_2$ is the simplicial complex $K_1 \ast K_2$, whose simplices are those subset 
		 $\sigma \subseteq V(K_1) \cup V(K_2)$ such that $\sigma \cap V(K_1) \in K_1$ and $\sigma \cap V(K_2) \in K_2$.
		\end{definition}
	
	In this article we consider a simplicial complex as a topological space, namely its  geometric realization. For  definition of geometric realization and details about simplicial complexes, we refer to the book \cite{dk} by  Kozlov.

		A topological space $X$ is said to be $k$-{\it connected} if the homotopy groups $\pi_m(X)$ are trivial for each $m \in \{0,1,\dots,k\}$.

For a space $X$, let ${\Sigma}^r(X)$ denote its $r$-fold suspension, where $r\geq 1$ is a natural number. Recall that, there is a homotopy equivalence
 
 \begin{equation}\label{remark:join of space with spheres is suspension}  \mathbb{S}^{r-1} \ast X \simeq \Sigma^{r}(X).
\end{equation}
 If $X$ is empty, then  $\Sigma(X) \simeq \bS^{0}$.
 The following results will be used repeatedly in this article.
 
	\begin{lemma}\cite[Lemma 10.4 (ii)]{bjorner}\label{lem:wedge}
		Let $K = K_0 \cup K_1 \cup \ldots \cup K_n $ be a simplicial complex with subcomplexes $K_i$ and assume that $K_i \cap K_j \subseteq K_0$ for all $ 1 \leq i < j \leq n$. If $K_i$ is contractible for all $0 \leq i \leq n$, then
		$$
		K \simeq \bigvee\limits_{i=1}^{n} \Sigma(K_i \cap K_0). 
		$$
	\end{lemma}
	
		The {\it nerve} of a family of sets $(A_i)_{i \in I}$  is the simplicial complex $\mathbf{N} = \mathbf{N}(\{A_i\})$ defined on the vertex set $I$ so that a finite subset $\sigma \subseteq I$ is in $\mathbf{N}$ precisely when $\bigcap\limits_{i \in \sigma} A_i \neq \emptyset$.

	\begin{theorem}\cite[Theorem 10.6(ii)]{bjorner} \label{thm:nerve}
		Let $K$ be a simplicial complex and $(K_i)_{i \in I}$ be a family of subcomplexes such that $K = \bigcup\limits_{i \in I} K_i$.
	Suppose every nonempty finite intersection $K_{i_1} \cap \ldots \cap K_{i_t}$ is $(k-t+1)$-connected. Then $K$ is $k$-connected if and only if  $\mathbf{N}(\{K_i\})$ is $k$-connected.
	
	\end{theorem}
	%\section{Proof of results}
	\section{Proof of \Cref{theorem:main_domination}}\label{sec:dominationnumber}
	Throughout this section, we fix the graph $G$ and an edge $e = \{u, v\} \in E(G)$. 
	%Let $\bar{e}$ denote the $1$-dimensional simplicial complex on vertex set $\{u, v\}$.
	Let $G -e$ denote the graph obtained from $G$ by removing $e$, {\itshape i.e.},  $V(G - e) = V(G)$ and $E(G-e) = E(G) \setminus \{e\}$. For a set $A$, let $\overline{A}$ denote the  simplex on vertex set $A$. For $0\leq i ,j \leq r-1$, define
	
	\begin{equation*}
	\begin{split}
	\S_{i,j}^{u,v} & : = \{ (S,T) : S \in \supp_i(u,G), T \in \supp_j(v,G) \mathrm{~and~} (G-e)[S\cup T \cup \{u,v\}] \mathrm{~is~not~connected}\}\\
	\mathrm{and} \\
	\Delta_{u, v} & : = \bigcup\limits_{i,j \leq r-1}\bigcup\limits_{ \substack{(S,T) \in \S_{i,j}^{u,v} }}\overline{S\cup T} \ast \ind_r(G-N[S \cup T \cup \{u,v\}]).\\
	\end{split}
	\end{equation*}
	
	Observe that, for $(S,T)\in \S_{i,j}^{u,v}$ we have that $(G-e)[S\cup T \cup \{u,v\}] = G[S \cup u\}] \sqcup G[T \cup \{v\}]$ which implies $S\cap T=\emptyset $ and $G[S\sqcup T] = G[S] \sqcup G[T]$.
	\begin{lemma}\label{lemma:union} For all $r\geq 1$, we have
	\begin{itemize}
		\item[($i$)]$\ind_r(G-e) = \ind_r(G) \cup  (\bar{e}\ast \Delta_{u, v})$.
		\item[($ii$)] 	$\ind_r(G) \cap  (\bar{e}\ast \Delta_{u, v}) = \big{(} \bigcup\limits_{i+j \leq r-2}   \bigcup\limits_{(S,T)\in \S_{i,j}^{u,v}} \bar{e} \ast (\overline{S \cup T})\ast \ind_r(G- N[S\cup T \cup \{u,v\}]) \big{)} \bigcup $ \\ \hspace*{3.85cm}$\big{(}  \bigcup\limits_{i,j \leq r-1} \bigcup\limits_{(S,T)\in \S_{i,j}^{u,v}} \bd(\bar{e}) \ast (\overline{S \cup T}) \ast \ind_r(G- N[S\cup T \cup \{u,v\}]) \big{)}$.
		\end{itemize}
	\end{lemma}
	\begin{proof}
	\begin{itemize}
	    \item[($i$)]
		Clearly, $\ind_r(G) \subseteq \ind_r(G-e)$. Since, $|S\cup \{u\}| \leq r, |T \cup \{v\}| \leq r$ and $S \cup T \cup \{u, v\}$ is not connected, we conclude that 
		$\bar{e} \ast \Delta_{u, v} \subseteq \ind_r(G-e)$. So, $\ind_r(G) \cup  (\bar{e}\ast \Delta_{u, v}) \subseteq \ind_r(G-e)$. Now, let $\sigma \in \ind_r(G-e)$. If  $ u \notin  \sigma$ or $v \notin \sigma$, then clearly $\sigma \in \ind_r(G)$. So, assume that $\{u, v\} \subseteq \sigma$. Let $C_u$ and $C_v$ be the connected components of $(G - e)[\sigma]$ containing $u$ and $v$ respectively. Observe that either $C_u = C_v$ or $V(C_u )\cap V(C_v) = \emptyset$. If $C_u = C_v$, then since $|V(C_u)| = |V(C_v)| \leq r$, we conclude that $\sigma \in \ind_r(G)$. If $C_v \neq C_u$, then take $S = V(C_u) \setminus \{u\}$ and $T = V(C_v) \setminus \{v\}$. Clearly, $(\sigma \setminus (V(C_u \cup C_v) )) \cap N[V(C_u \cup C_v)] = \emptyset$, which implies that $\sigma = V(C_u \cup C_v) \cup \tau = \{u,v\} \sqcup (S \cup T) \sqcup \tau $ for some $\tau \in \ind_r(G - N[V(C_u \cup C_v)]) = \ind_r(G- N[S \cup T \cup \{u, v\}])$. 
	
%	\begin{claim}\label{claim:intersection}
	
%	\end{claim}
\item[($ii$)]For simplicity of notation, let
\begin{align*}  Z_1=& \bigcup\limits_{i+j \leq r-2} \bigcup\limits_{(S,T)\in \S_{i,j}^{u,v}} \bar{e} \ast (\overline{S \cup T})\ast \ind_r(G- N[S\cup T \cup \{u,v\}]) \mathrm{~and~} \\
Z_2 = & \bigcup\limits_{i,j \leq r-1} \bigsqcup\limits_{(S,T)\in \S_{i,j}^{u,v}} \bd(\bar{e}) \ast (\overline{S \cup T}) \ast \ind_r(G- N[S\cup T \cup \{u,v\}]). 
\end{align*}

Since $S \cap T = \emptyset,$  $ i +j \leq r-2$ implies that $|S \cup T| \leq r-2$.
Hence, $\bar{e}  \ast (\overline{S \cup T}) \in \ind_r(G) \cap (\bar{e} \ast \Delta_{u, v})$. Therefore, $Z_1 \subseteq \ind_r(G) \cap \bar{e} \ast \Delta_{u, v}$. Now, let $i,j \leq r-1$ and $(S,T)\in \S_{i,j}^{u,v}$. Since $G[S\cup T \cup \{u\}]=G[S\cup \{u\}] \sqcup G[T]$ and $G[S\cup T \cup \{v\}]=G[S] \sqcup G[T\cup \{v\}]$, we get that $\{u\}\sqcup S \sqcup T,~ \{v\}\sqcup S \sqcup T \in \ind_r(G)$. Therefore, we conclude that $Z_2 \subseteq \ind_r(G) \cap (\bar{e} \ast \Delta_{u, v})$. To show the other way inclusion, let $\sigma \in  \ind_r(G) \cap (\bar{e} \ast \Delta_{u, v})$. There exist $i, j \leq r-1$ and $(S, T) \in \S_{i,j}^{u,v}$ such that $\sigma \setminus  \{u,v\} \in \overline{S \cup T} \ast \ind_r(G - N[S \cup T \cup \{u, v\}])$. 
		
		{\bf Case 1.} $\{u, v\} \subseteq \sigma$. 
		
		 Write $\sigma  =  \{u, v\} \sqcup \tau \sqcup \gamma $, where $\tau = \sigma \cap (S \cup T)$ and $\gamma = \sigma \setminus (\{u, v\} \cup \tau ) \in \ind_r(G - N[S \cup T \cup \{u, v\}])$. Let $C$ be the connected component of $G [\sigma]$ containing $u$ (and hence $v \in C$). Observe that, $C \setminus \{u, v\} \subseteq \tau$. Now, let $S_1 = V(C) \cap S$ and $T_1 = V(C) \cap T$.  Clearly, $(S_1, T_1) \in \S_{i_1, j_1}^{u, v}$ for some $i_1 \leq i, j_1 \leq j$. Since $V(C)\subseteq \sigma \in \ind_r(G),~|V(C)| \leq r$ and therefore $i_1 + j_1 \leq r-2$. Further, since $C$ is a component,   $(\sigma \setminus V(C)) \cap N[V(C)] =\emptyset$ and therefore  $\sigma \setminus V(C) \in \ind_r(G- N[V(C)]) =
		\ind_r(G- N[S_1 \cup T_1 \cup \{u, v\}])$. Thus,  we conclude that  $\sigma \in Z_1$. 
		
		{\bf Case 2.} $\{u, v\} \not\subseteq \sigma$
		
	Since $\sigma \in \bar{e} \ast \Delta_{u, v}$ and $\{u, v\} \not\subseteq \sigma $, we get that $\sigma \in \bd(\bar{e}) \ast (\overline{S \cup T}) \ast \ind_r(G - N [S \cup T \cup \{u, v\}]) $ and therefore $\sigma \in Z_2$.
	\end{itemize}
	\end{proof}

	\begin{proposition}\label{prop:conditionforclaim4} 
	Let $t$ be a positive integer and  	for each $1 \leq l \leq t$, let $S_l$ and $T_l$ be supports of $u$ and $v$ respectively. For each $1 \leq l \leq t$,  let there exists $i_l, j_l \leq r-1$, such that $(S_l, T_l)\in \S_{i_l, j_l}^{u, v}$. Let $\mathcal{L} = (\bigcup\limits_{l=1}^t (S_l \cup T_l)) \setminus (\bigcup\limits_{l=1}^t (N[L_l]\setminus L_l))$, where $L_l = S_l \cup T_l \cup \{u,v\}$. Then $\bd(\bar{e}) \ast \overline{\mathcal{L}}  \subseteq \ind_r(G)$.
	\end{proposition}

	\begin{proof} Let $\sigma \in \text{Bd}(\bar{e})\ast \overline{\mathcal{L}}$. Observe that 
		for any $l$ and  $w \in N[L_l] \setminus L_l, w \notin \sigma$. Recall that, $(G-e)[L_l]= (G-e)[S_l\cup\{u\}] \sqcup (G-e)[T_l\cup \{v\}]=G[S_l\cup\{u\}] \sqcup G[T_l\cup \{v\}]$. Therefore, any connected component of $G[\sigma]$ must be a subset of either $S_i\cup\{u\}$ for some $1\leq i \leq t$ or $T_j \cup \{v\}$ for some $1\leq j \leq t$. Since $|S_i|, |T_j| \leq r-1$, the result follows.
	\end{proof}
	
		For $0\leq i,j \leq r-1$ and for any $S \in \supp_i(u,G)$ and $T \in \supp_j(v,G)$, let 
		%\begin{equation*}
		\begin{align*}
		%L_{S, T} & = S\cup T \cup \{u,v\},\\
		W_{S,T} & = \partial(\bar{e})\ast \overline{S\cup T} \ast \ind_r(G-N[S \cup T \cup \{u, v\}]) \mathrm{~and~} \\
		Y_{S,T} & = \bar{e}\ast \overline{S\cup T} \ast \ind_r(G-N[S \cup T \cup \{u, v\}]).
		\end{align*} 
		%\end{equation*}

		Using \Cref{lemma:union}$(ii)$, we can write $\ind_r(G) \cap  (\bar{e}\ast \Delta_{u, v}) = \bigcup\limits_{l} X_l$, where each $X_l$ is of the form either $W_{S, T}$ or $Y_{S, T}$ for some 
		$(S, T) \in \S_{i, j}^{u, v}$. We first understand the structure of arbitrary $t$-intersection of $X_i$'s, {\itshape i.e.}, of $X_{i_1} \cap \ldots \cap X_{i_t}$.  For each $1 \leq l \leq t$, there exists $(S_{l}, T_{l})$ such that either $X_{i_l} = W_{S_{l}, T_{l}}$ or $X_{i_l} = Y_{S_{l}, T_{l}}$.

		\begin{lemma}\label{lemma:generalintersection}
			%\begin{equation}\label{equation:intersection}
			For $1 \leq l \leq t$, let $L_{l}$ and $\mathcal{L}$ be as in \Cref{prop:conditionforclaim4}. Then $ \bigcap\limits_{l=1}^t X_{i_l}$ is either $ \bd(\bar{e}) \ast \ind_r(G-N[\bigcup\limits_{l=1}^t L_{l}])\ast \overline{\mathcal{L}}$ or a cone over $\overline{e}$.
			%\end{equation}
		\end{lemma}
		
			\begin{proof} 
			If each $X_{i_l}$ is of the form $W_{S_l, T_l}$, then clearly since $X_{i_1} \cap \ldots \cap X_{i_t}$ is a cone over $\bar{e}$ and therefore it is contractible. So assume that there exists at least one $1 \leq l \leq t$ such that $X_{i_l} $  is of the form $Y_{S_l, T_l}$. In this case, we show the following.
			
			$$ \bigcap\limits_{l=1}^t X_{i_l}= \bd(\bar{e}) \ast \ind_r(G-N[\bigcup\limits_{l=1}^t L_{l}])\ast \overline{\mathcal{L}}.$$
			
			Let $\sigma \in \bigcap\limits_{l=1}^t X_{i_l}$. Define $\sigma = \tau_1 \sqcup \tau_2 \sqcup \tau_3$, where $\tau_1 = \sigma \cap \{u, v\}$, $\tau_2 = \sigma \cap V(G - N[\bigcup\limits_{l=1}^t L_{l}])$ and $\tau_3 = \sigma \setminus (\tau_1 \cup \tau_2)$. To prove that $\sigma \in \bd(\bar{e}) \ast \ind_r(G-N[\bigcup\limits_{l=1}^t L_{l}])\ast \overline{\mathcal{L}}$, it is enough to show that $\tau_3 \subseteq \mathcal{L} $. Let $L = \bigcup\limits_{l=1}^t L_{l}$. Observe that $\tau_3 \subseteq  N[L] \setminus \{u, v\}$. Suppose  there exists $w \in \tau_3$ such that $w \in  N[L] \setminus L$. Then there exist $x \in L$ such that $x \sim w$. 
			Firtly, if $x \in \{u, v\}$, then $w \in N[u, v] \setminus L$, which implies $w \in N[S_{1} \cup T_{1} \cup \{u, v\}] \setminus L \implies \{w\} \notin X_{i_1}$, a contradiction. Secondly, if $x \in \bigcup \limits_{l=1}^t (S_{l} \cup T_{l})$, then without loss of generality we assume that $x \in S_{1} \cup T_{1}$. Here, $w \in N[L_1] \setminus L \implies \{w\} \notin X_{i_1}$, again a contradiction. Therefore $\tau_3 \subseteq L\setminus \{u,v\} = \bigcup\limits_{l=1}^t (S_{l} \cup T_{l})$. We now show that 
			$\tau_3 \cap \bigcup\limits_{l=1}^t (N[L_{l}] \setminus L_{l})  = \emptyset$. Let $z \in \tau_3 \cap \bigcup\limits_{l=1}^t N[L_{l}] \setminus L_{l}$. If $z \in N[L_{l}] \setminus L_{l}$, then $\{z\} \notin X_{i_l}$, which is a contradiction. Hence, $\tau_3 \subseteq \mathcal{L}$.
			
			To show the other way inclusion, let $\sigma \in \bd(\bar{e}) \ast \ind_r(G-N[\bigcup\limits_{l=1}^t L_{l}])\ast \overline{\mathcal{L}}$. From \Cref{prop:conditionforclaim4}, $\sigma \in \ind_r(G)$. Write $\sigma=\sigma_1\sqcup \sigma_2 \sqcup \sigma_3$, where $\sigma_1 = \sigma \cap \{u,v\}$, $\sigma_2  = \sigma \cap V(G-N[\bigcup\limits_{l=1}^t L_{l}])$ and $\sigma_3 \subseteq \mathcal{L}$. Again, write $\sigma_3 =\sigma_3^1 \sqcup \sigma_3^2$, where $\sigma_3^1 = \sigma_3 \cap (S_{1}\cup T_{1})$ and $\sigma_3^2=\sigma_3\setminus \sigma_3^1$. Clearly, $\sigma_3^2 \subseteq \bigcup\limits_{l=2}^t (S_{l} \cup T_{l})$ and $\sigma_3^2 \cap N[L_{1}] = \emptyset$, which implies  that $\sigma_3^2 \in \ind_r(G-N[L_{1}])$.
			
			Since $\sigma_3^2 \subseteq \bigcup\limits_{l=2}^t (S_{l} \cup T_{l}) \setminus N[L_{1}]$ and $\sigma_2 \cap  N[\bigcup\limits_{l=1}^t L_{l}] =\emptyset$, we conclude that $\sigma_3^2 \cup \sigma_2 \in \ind_r(G-N[L_{1}])$. Thus, $\sigma$ can be written as follows: $\sigma= \sigma_1 \sqcup \sigma_2 \sqcup \sigma_3^1 \sqcup \sigma_3^2 = \sigma_1 \sqcup \sigma_3^1 \sqcup (\sigma_2 \sqcup\sigma_3^2) \in X_{i_1}$. By similar arguments, we see that $\sigma \in X_{i_l}$ for each $ 1\leq l\leq t$. 
			%This completes the proof of \Cref{lemma:generalintersection}.
		\end{proof}
		
		\begin{proposition} \label{prop8}(The Hurewicz Theorem)\\
 If a space $X$ is $(n - 1)$ connected, $n \geq 2$, then $\tilde{H_i} (X) = 0$ for $i < n$
and $\pi_n (X) \cong H_n (X)$.
\end{proposition}
	
	We are now ready to prove the main result of this section.

		\begin{theorem}\label{thm:maindominationinside}
		Let $G$ be a graph and $r \geq 1$. 
		\begin{itemize}
			\item[(i)] If $\gamma_r(G) > 2k$, then $\tilde{H}_{j}(\ind_r(G)) = 0$ for all $j \leq k+r-2$. 
			\item[(ii)] If $\omega_r(G) > 2k$, then $\tilde{H}_{j}(\ind_r(G)) = 0$ for all $j \leq k-1$. 
			
		\end{itemize}
	\end{theorem}
	
	\begin{proof}
	If $r = 1$, then since		
 $\gamma_1(G) = \omega_1(G) = \Gamma_0(G)$, result follows from \Cref{thm:meshulam}$(i)$. So, assume that $r \geq2$. 
 	From \Cref{lemma:union}$(i)$ and using Mayer–Vietoris sequence, we have 
		
		\begin{align*}
		\cdots \to \tilde{H}_i(\ind_r(G) \cap (\bar{e} \ast \Delta_{u, v})) \to \tilde{H}_i(\ind_r(G) ) \oplus \tilde{H}_i(\bar{e} \ast \Delta_{u, v})
		\to \tilde{H}_i(\ind_r(G-e)) \to \cdots 
		\end{align*}
		
		Since $\bar{e} \ast \Delta_{u, v}$ is contractible, we get the following sequence.
		
		\begin{align}\label{MVsequence}
		\cdots \to \tilde{H}_i(\ind_r(G) \cap (\bar{e} \ast \Delta_{u, v})) \to \tilde{H}_i(\ind_r(G) ) \to \tilde{H}_i(\ind_r(G-e)) \to \cdots 
		\end{align}
		
	The remaining proof is now by induction on $k$ and the number of edges of $G$. Using \Cref{lemma:union}$(ii)$, we get that $\ind_1(G) \cap  (\bar{e}\ast \Delta_{u, v}) =  \bd(\bar{e}) \ast  \ind_r(G- N[\{u,v\}]) = \Sigma(\ind_r(G- N[\{u,v\}]))$

			%Since $\ind_r(G)$ is always $(r-2)$-connected whenever $G$ is non-empty, the base case  $k= 0$ follows trivially. 
			
			%Inductively assume that for any graph $H$ and $l < k$,  if $\gamma_r(H) > 2l$, then $\tilde{H}_{l+r-2}(\ind_r(H)) = 0$ and  
		%	if  $\omega_r(H) > 2l$, then $\tilde{H}_{l-1}(\ind_r(H)) = 0$. 

		\begin{itemize} \item[($i$)]
		
		Since $\ind_r(G)$ is always ($r-2$)-connected for any non empty graph,  base case of induction follows.  So, assume that $k \geq 1$.  Clearly, 
			$\gamma_r(G-e) \geq \gamma_r(G)$ and hence by induction $\tilde{H}_{k+r-2}(\ind_r(G-e)) = 0$
		
		\begin{claim}\label{claim:connectivity1}  $\ind_r(G) \cap  (\bar{e}\ast \Delta_{u, v})$ is   ($k+r-2$)-connected.
		\end{claim}
		\begin{proof}[Proof of \Cref{claim:connectivity1}] 	Using \Cref{lemma:union}$(ii)$, we can write $\ind_r(G) \cap  (\bar{e}\ast \Delta_{u, v}) = \bigcup\limits_{l} X_l$, where each $X_l$ is of the form either $W_{S, T}$ or $Y_{S, T}$ for some 
		$(S, T) \in \S_{i, j}^{u, v}$. Consider the intersection $X_{i_1} \cap \ldots \cap X_{i_t}$.  For each $1 \leq l \leq t$, there exists $(S_{i_l}, T_{i_l})$ such that either $X_{i_l} = W_{S_{i_l}, T_{i_l}}$ or $X_{i_l} = Y_{S_{i_l}, T_{i_l}}$.  
		Observe that  $\gamma_r(G-N[\bigcup\limits_{l=1}^t (S_{i_l} \cup T_{i_l} \cup \{u, v\})]) \geq \gamma_r(G) - 2t > 2(k-t).$ Hence, by  induction
			\begin{equation}\label{equation:conditionforintersectiontobesimplyconnected}
			\tilde{H}_{j} (\ind_r(G- N[\bigcup\limits_{l=1}^t (S_{i_l} \cup T_{i_l} \cup \{u, v\})])) = 0, \ \forall \ j \leq k-t+r+2. 
			\end{equation}
			
			Since $\gamma_r(G)>2$, $G-N[\bigcup\limits_{l=1}^t (S_{i_l} \cup T_{i_l} \cup \{u, v\})]$ is not an empty graph. Further, $r\geq 2$ implies that $\ind_r(G-N[\bigcup\limits_{l=1}^t (S_{i_l} \cup T_{i_l} \cup \{u, v\})])$ is path connected. Since the join of a path connected space with an non empty space is always simply connected, using  \Cref{lemma:generalintersection} we conclude that  $\bigcap\limits_{l=1}^t X_{i_l}$ is simply connected.  Hence, from equations \Cref{equation:conditionforintersectiontobesimplyconnected} and \Cref{lemma:generalintersection}, we get $\tilde{H}_j(\bigcap\limits_{l=1}^t X_{i_l}) = 0 \ \forall \ j \leq k+r-t-1$. Hence, $\bigcap\limits_{l=1}^t X_{i_l}$ is $(k+r-t-1)$-connected by \Cref{prop8}. It is easy to check that the nerve  $\mathbf{N}(\{X_{l}\})$ is a simplex and hence contractible. Therefore, result follows from \Cref{thm:nerve}.
		\end{proof}

			  Since $\tilde{H}_{j}(\ind_r(G-e)) = 0 \ \forall \ j \leq k+r-2$, from \Cref{claim:connectivity1} and Equation (\ref{MVsequence}), we get that $\tilde{H}_{j}(\ind_r(G)) = 0 \ \forall \ j \leq k+r-2$.

			\item[$(ii)$]  Again, the base case is straight forward. Since $\omega_r(G-e) \geq \omega_r(G)$, by induction $\tilde{H}_{j}(\ind_r(G-e)) = 0$ for all $j\leq k-1$.
			
			\begin{claim}\label{claim:connectivity2}
			$\ind_r(G) \cap  (\bar{e}\ast \Delta_{u, v})$ is  ($k-1$)-connected.
			\end{claim}
			\begin{proof}[Proof of \Cref{claim:connectivity2}]
		The proof here is similar to that of \Cref{claim:connectivity1}. Again, write  $\ind_r(G) \cap  (\bar{e}\ast \Delta_{u, v}) = \bigcup\limits_{l} X_l$ and  consider the intersection $X_{i_1} \cap \ldots \cap X_{i_t}$, where each $X_{i_l}$ is  either $W_{S_{i_l}, T_{i_l}}$ or $Y_{S_{i_l}, T_{i_l}}$.  
			 	For each $1 \leq l \leq t$, let $L_{i_l} = S_{i_l} \cup T_{i_l} \cup \{u, v\}$.	Observe that if $\mathcal{D}_r$ is a dominating $r$-collection (see \Cref{def:dominate}) for $G - N[\bigcup\limits_{l=1}^t L_{i_l}]$, then $\mathcal{D}_r \cup \{S_{i_l}\cup \{u\}, T_{i_l}\cup \{v\} : 1 \leq l \leq t\}$ is a dominating $r$-collection for $G$. Hence, $\omega_r(G-N[\bigcup\limits_{l=1}^t L_{i_l}]) \geq \omega_r(G) - 2t > 2(k-t).$
			Hence, by induction 
			\begin{equation}\label{equation:conditionfor intersectiontobesimplyconnected}
			\tilde{H}_{j} (\ind_r(G- N[\bigcup\limits_{l=1}^t L_{i_l}])) = 0 ~ \mathrm{ for~ all~} j \leq k-t-1. 
			\end{equation}
			
		 If  $G-N[\bigcup\limits_{l=1}^t L_{i_l}]$ is non empty graph, then  $r \geq 2$ implies that  $\ind_r(G-N[\bigcup\limits_{l=1}^t L_{i_l}])$ is path connected. Therefore, from \Cref{lemma:generalintersection}, $\bigcap\limits_{l=1}^t X_{i_l}$ is simply connected.  Hence, from \Cref{lemma:generalintersection}, \Cref{equation:conditionfor intersectiontobesimplyconnected} and \Cref{prop8}, we get that $(\bigcap\limits_{l=1}^t X_{i_l})$ is $(k-t)$-connected.
			
			If $G-N[\bigcup\limits_{l=1}^t L_{i_l}]$ is  empty then observe that 
			$\omega_r(G) \leq 2t$ and therefore $k-t < 0$. Since $\bd(\bar{e}) \subseteq \bigcap\limits_{l=1}^t X_{i_l}$, we see that  $\bigcap\limits_{l=1}^t X_{i_l}  \neq \emptyset$. Thus, we  conclude that  $\bigcap\limits_{l=1}^t X_{i_l}$ is $k-t \leq -1$ connected. 
			
			Since the nerve  $\mathbf{N}(\{X_{l}\})$ is contractible and arbitrary $t$ intersection $\bigcap\limits_{l=1}^{t} X_{i_l}$ is $k-t$ connected, \Cref{thm:nerve} implies that $\ind_r(G) \cap  (\bar{e}\ast \Delta_{u, v})$ is  ($k-1$)-connected.
			\end{proof}
			Since $\tilde{H}_{j}(\ind_r(G-e)) = 0 \ \forall \ j \leq k-1$, from \Cref{claim:connectivity2} and Equation (\ref{MVsequence}), we get that $\tilde{H}_{j}(\ind_r(G))=0$ for all $j \leq k-1$.
		\end{itemize}
		\vspace{-0.87 cm}
	\end{proof}

	\section{Chordal graphs} \label{sec:chordal}
	In this section, we study $r$-independence complexes of chordal graphs. However, first we present a general result that relates the $r$-independence complex of a graph $G$ with  $r$-independence complexes of its certain proper subgraphs (cf. \Cref{thm:main theorem for support}).
	
	\begin{lemma} \label{lem:support}
		Let $G$ be a graph, $v \in V(G)$ and let $\supp_r(v,G) = \{S_1, \ldots, S_n\}$. Then  $$\ind_r(G)=\st(v) \cup \bigcup\limits_{i=1}^n \st (S_i).$$
	\end{lemma}
	
	\begin{proof}
		Since $\st(v),~ \st(S_i) \subseteq \ind_r(G)$ for all $i \in \{1,\dots,n\}$, $\st(v) \cup \bigcup\limits_{i}^n \st (S_i) \subseteq \ind_r(G)$. To show the other way inclusion, let $\sigma \in \ind_r(G)$. If $\sigma \notin \st(v)$, then $G[\sigma \cup \{v\}]$ has a connected component with at least $r+1$ vertices. Let $H$ be a such connected component of  $G[\sigma\cup \{v\}]$. Since $\sigma \in \ind_r(G)$, $v \in V(H)$. Choose a subset $S \subset V(H) \setminus \{v\}$,  such that $|S| = r$ and $G[S \cup \{v\}]$ is connected.  Then $S \in \mathrm{Supp}_r(v,G)$ and $\sigma \in \st(S)$.
	\end{proof}
	
	Recall that for  a vertex $v\in V(G)$, $\supp_r(v,G)$ is called connected if $G[S]$ is connected for all $S \in \supp_r(v,G)$.
	\begin{theorem}\label{thm:main theorem for support}
		Let $G$ be a connected graph, $v \in V(G)$ and $\mathrm{Supp}_r(v,G)=\{S_1,S_2,\dots,$ $S_n\}$. If $\supp_r(v,G)$ is connected and $N(v) \subseteq N[S_i]$ for each $i \in \{1,\dots,n\}$, then 
		$$\ind_r(G) \simeq {\bigvee\limits_{i=1}^n} \Sigma^r (\ind_r(G-N[S_i])) .$$
	\end{theorem}
	
	\begin{proof}
		
		From Lemma \ref{lem:support},  $\ind_r(G)=\st(v) \cup \bigcup\limits_{i}^n \st (S_i).$
		
		\begin{claim}\label{claim:main theorem claim 1}
			For all $1 \leq i < j \leq n$, 	$\st(S_i) \cap \st(S_j) \subseteq \st(v)$. 
		\end{claim}
		\begin{proof}[Proof of Claim \ref{claim:main theorem claim 1}]
			
			Fix $i \neq j$ and let  $\sigma \in \st(S_i) \cap \st(S_j)$. Clearly, $v \notin \sigma$. Since $\supp_r(v,G)$ is connected, we see that $\sigma \cap (N[S_i] \setminus S_i) = \emptyset$ and $\sigma \cap (N[S_j] \setminus S_j) = \emptyset$. Hence, $N(v) \subseteq N[S_i] \cap N[S_j]$ implies that $\sigma \cap N(v) \subseteq S_i \cap S_j$.
			Suppose $\sigma \notin \st(v)$.  Let $\tau = \sigma \cap N(v)$ and $H$ be a connected component of $G[\sigma \cup \{v\}]$ of cardinality  at least $r+1$. Clearly, $(V(H) \setminus \{v\}) \nsubseteq  (S_i\cap S_j)$ (since $|S_i\cap S_j| <r$). Let $w \in (V(H)\setminus \{v\}) \setminus (S_i\cap S_j)$ and $P$ be a path of minimal length from $w$ to $v$ in $H$. Without loss of generality, we can assume that $w\notin S_i$. Note that, $G[S_i]$ and $G[(V(P)\setminus \{v\})]$ are connected subgraphs. Further, $N(v) \subseteq N[S_i]$ implies that   $N[S_i] \cap (V(P)\setminus \{v\}) \neq \emptyset$. Therefore, $G[S_i \cup (V(P)\setminus \{v\})]$ is connected subgraph of cardinality more than $r$. Hence, $V(P)\setminus \{v\} \notin \mathrm{st}(S_i)$ implying that $\sigma \notin \mathrm{st}(S_i)$. Which is a contradiction to our assumption that $\sigma \in \st(S_i) \cap \st(S_j)$.
		\end{proof}
		
		For each $1 \leq i \leq n$, let $\Delta^{S_i}$ be the simplex on vertex set $S_i$.
		\begin{claim}\label{claim:main theorem claim2}
			For each $i \in \{1,\dots,n\}$, $\st(S_i) \cap \st(v) = \ind_r(G-N[S_i]) \ast \bd(\Delta^{S_i}).$
		\end{claim}
		
		\begin{proof}[Proof of Claim \ref{claim:main theorem claim2}] 
			%		Let $\sigma \in  \ind_r(G-N[S_i])$.  Clearly, $\sigma \in \st(S_i)$. Since $S_i \cap N(v) \neq \emptyset$ and $v \in N(S_i)$, we see that $\{v\} \cup N(v) \subset N[S_i]$. Hence $\sigma \cap (\{v\} \cup N(v)) = \emptyset$ thereby showing that $\sigma \in \st(v)$. Therefore $\ind_r(G - N[S_i]) \subset \st(S_i) \cap \st(v)$. 
			%		
			Let $\tau = \sigma \sqcup \delta$, where $\sigma \in \ind_r(G - N[S_i])$ and $\delta \subsetneq S_i $. Since $\sigma \cap N[S_i] = \emptyset$ and $N(v) \subseteq N[S_i]$, we get that $\tau \in \st(S_i)$ and $v \notin N(\sigma) $. Therefore, $N(\sigma) \cap (\{v\} \cup \delta) = \emptyset $ which implies that $\tau \in \st(v)$. Hence, $\ind_r(G-N[S_i]) \ast \bd(\Delta^{S_i}) \subseteq  \st(S_i) \cap \st(v)$. 
			
			To show the other way inclusion, let $\sigma \in \st(S_i) \cap \st(v)$. Since $\sigma \in \st(v)$ and $S_i$ is an $r$-support, we see that $S_i \nsubseteq \sigma$. Further, $\sigma \in \st(S_i)$ implies that $\sigma \cap (N[S_i] \setminus S_i) = \emptyset$. Let  $\tau = \sigma \cap (V(G) \setminus N[S_i])$ and $\delta = \sigma \cap S_i$. Then $\sigma = \tau \sqcup \delta$ and therefore $\sigma \in \ind_r(G-N[S_i]) \ast \bd(\Delta^{S_i})$. 
		\end{proof} 
		
		Observe that  $\ind_r(G-N[S_i]) \ast \mathrm{Bd}(\Delta^{S_i})\simeq \Sigma^{r-1}(\ind_r(G-N[S_i]))$ and $\st(v)$,  $\st(S_i)$ is contractible. Therefore, \Cref{thm:main theorem for support} follows from \Cref{claim:main theorem claim 1}, \Cref{claim:main theorem claim2}, \Cref{lem:wedge}, and \Cref{lem:support}.
	\end{proof}
	
	For  $n \geq 3$, a {\itshape wheel graph}, denoted $W_n$, is a graph constructed from the cycle graph $C_n$ by adding a new vertex $w$ and an edge $(w, v)$ for each vertex $v \in V(C_n)$. As an immediate consequence of \Cref{thm:main theorem for support} we get the following results.
	\begin{cor}
		For $n\geq 3$,	$\ind_{n-1}(W_n) \simeq \bigvee\limits_{n} \mathbb{S}^{n-2}$.
	\end{cor}
	\begin{proof}
		It is easy to observe that $\supp_{n-1}(w,W_n)$ is connected, $|\supp_{n-1}(w,W_n)|=n$ and $W_n-N[S] =\emptyset$ for each $S \in \supp_{n-1}(w,W_n)$. Therefore, $\ind_{n-1}(W_n) \simeq \bigvee\limits_{n} \Sigma^{n-1} (\emptyset) = \bigvee\limits_{n} \mathbb{S}^{n-2}$.
	\end{proof}
	
	For a vertex $v$ of  graph $G$, let $N_2(v)$ denotes the set of vertices of $G$ whose distance from $v$ is exactly $2$, {\itshape i.e.}, $N_2(v)= \{w \in V(G) : d(v,w)=2\}$. 
	\begin{cor}\label{cor:notchordalsuspension}
		Let $v\in V(G)$ such that $N(v) \times N_2(v) \subseteq E(G)$. If $r > |N(v)|$, then $\supp_r(v)$ is connected. In particular, $\tilde{H}_i(\ind_r(G))=0$ for all $i<r-1$ and $\tilde{H}_j(\ind_r(G))$ is torsion-free for $j=r-1,r$.
	\end{cor}
	
	% 	\subsection{Chordal graph}
	
	A vertex $v$ of a graph $G$ is called {\it simplicial}, if the induced subgraph $G[N(v)]$ is a complete graph. It is a  classical result of Dirac \cite{dirac} that every  chordal graph has a simplicial vertex. For a simplicial vertex $v$ and $S\in \supp_r(v,G)$, since each connected component of $G[S]$ has a vertex from $N(v)$, we get the following.
	
	\begin{remark}\label{lemma:eachsupportconnected}
		If $v\in V(G)$ is a simplicial vertex, then $\supp_r(v,G)$ is connected and $N(v)\subseteq N[S]$ for each $S \in \supp_r(v,G)$.
	\end{remark}
	
	It is easy to see that, $\supp_1(v,G)$ is always connected and $\supp_2(v,G)$ is connected if and only if $v$ is a simplicial vertex. 
	%Observe that, if $N(v) \times N_2(v) \subseteq E(G)$ and  $r > |N(v)|$, (as given in  \Cref{cor:notchordalsuspension}), then we can construct examples in which $v$ is not simplicial but $\supp_r(v,G)$ is connected.
	When $r>2$, we can construct examples satisfying the condition given in \Cref{cor:notchordalsuspension} such that $v$ is not simplicial yet $\supp_r(v,G)$ is connected. Thus, the condition being $v$ simplicial is sufficient but not necessary for $\supp_r(v,G)$ to be connected. 
	
	We now compute the homotopy type of $r$-independence complexes of chordal graphs.
	
		\begin{theorem} \label{theorem:lowerboundconnectivityinside} Let $G$ be a chordal graph and $r \geq 1$.
		\begin{itemize}
			\item[(i)] $\ind_r(G)$ is either contractible or homotopy equivalent to wedge of spheres.
			
			% \item[(ii)] If $\ind_r(G)\simeq  \bigvee \mathbb{S}^{i_k}$, then $i_k \geq r\omega_r(G)-1$ and for each $k$ there exists an $s_k \in \mathbb{N}$ such that $i_k = rs_k-1$.

			\item[(ii)] If  $\omega_r(G) > k$, then $\tilde{H}_{i}(\ind_r(G)) = 0$ for each $i \leq rk-1$.
			\item[(iii)] If $\ind_r(G)\simeq  \bigvee \mathbb{S}^{i_k}$, then for each $i_k$ there exists a positive integer $s_k $ such that $i_k = rs_k-1$.
		\end{itemize}
	\end{theorem}
	\begin{proof} Let $v$ be a simplicial vertex of $G$ and let $\supp_r(v, G) = \{S_1, \ldots, S_n\}$. 
		\begin{itemize}
			\item[$(i)$] From \Cref{lemma:eachsupportconnected} and  \Cref{thm:main theorem for support}, we have 
			$$\ind_r(G) \simeq {\bigvee\limits_{i=1}^n} \Sigma^r (\ind_r(G-N[S_i])) .$$
			Hence, the proof follows from induction on number of vertices of graph and the fact that induced subgraph of a chordal graph is also chordal.

			\item[$(ii)$]  \Cref{lemma:eachsupportconnected} implies that $G[S_i]$ is connected for each $i \in \{1,\dots,n\}$ and using \Cref{thm:main theorem for support}, we get $\ind_r(G) \simeq {\bigvee\limits_{i=1}^n} \Sigma^r (\ind_r(G-N[S_i])) .$ If each $\ind_r(G-N[S_i])$ is contractible, then so is $\ind_r(G)$. Suppose there exists $i$ such that
			$\ind_r(G-N[S_i])$ is not contractible. Let $\ind_r(G-N[S_i]) \simeq \bigvee \mathbb{S}^{i_t}$. Then $\Sigma^{r}(\ind_r(G-N[S_i])) \simeq \bigvee \mathbb{S}^{i_t+r}$. For each dominating $r$-collection $D_r$ of $G-N[S_i]$, the set 
			$D_r \cup \{S_i\}$ is also a dominating $r$-collection of $G$. This implies $\omega_r(G) \leq \omega_r(G-N[S_i]) +1$. By induction, we have $i_t \geq r \omega_r(G-N[S_i]) -1$. Hence, $i_t+r \geq r(\omega_r(G)-1)-1+r = r\omega_r(G)-1$ for each $i_t$.
			
			\item[$(iii)$] If $|V(G)| \leq r$, then $\ind_r(G)$ is contractible and result follows trivially. Now, let  $|V(G)| \geq r+1$ and $v$ be a simplicial vertex of $G$. If $\supp_r(v, G) = \emptyset$, then the connected component say $H_v$ of $G$ containing $v$ is of cardinality at most $r$. In this case $\ind_r(G) \simeq \ind_r(G-H_v) \ast \ind_r(H_v)$. Since, $\ind_r(H_v)$ is contractible, $\ind_r(G)$ is also contractible. So, assume that $\supp_r(v, G) \neq \emptyset$.
			Using Theorem \ref{thm:main theorem for support}, we get $$\ind_r(G) \simeq \bigvee\limits_{S \in \supp_r(v,G)} \Sigma^r (\ind_r(G-N[S])).$$ 
			By induction, if $\ind_r(G-N[S]) \simeq \bigvee \mathbb{S}^{i_t}$ then for each $i_t$ there  exists an $s_t \in \mathbb{N}$ such that   $i_t = rs_t-1$. Hence, we get that $i_t+r = r(s_t+1)-1$.
		\end{itemize}
		\vspace{-0.75cm}
	\end{proof}
	
The following result is the converse part of \Cref{theorem:lowerboundconnectivityinside} $(iii)$.
	
	\begin{theorem}\label{thm:reverseconstructioninside}
		Let $r \geq 2$. Let $(d_1, \ldots, d_n)$ and $(k_1, \ldots, k_n)$ be two sequences of positive integers. There exists a chordal graph $G$ such that $\ind_r(G) \simeq \bigvee\limits_{i=1}^n \vee_{d_i} \mathbb{S}^{rk_i-1}$. 
	\end{theorem}
	Before proving \Cref{thm:reverseconstructioninside}, we illustrate the construction of desired graph  by an example. Fix $r=2,~ (d_1, d_2) = (1, 2)$ and   $(k_1,k_2)=(1,2)$. For $n \geq 1$, let $P_n$ denote the path graph on $n$ vertices. The $r$-independence complexes of path graphs have been computed by Paolini and Salvetti. They proved the following.
	
	\begin{proposition}[{\cite[Proposition 3.7]{PS18}}]\label{prop:higher_ind_path} For $r\geq 1$, we have
		\[\ind_{r}(P_n) \cong
		\begin{cases} 
		\mathbb{S}^{rk-1}, & \mathrm{if\ } n = (r+2)k \mathrm{\  or\ } n=(r+2)k-1;\\
		\{\mathrm{point}\}, & \mathrm{otherwise}.
		\end{cases}
		\]
	\end{proposition}

	Let $G$ be the graph given in \Cref{fig:reverse construction}. Clearly, $G$ is a chordal graph and $v_1$ is a simplicial vertex. Here, $\supp_2(v,G)= \{\{v_2,a_1\}, \{v_2,b_1\},\{v_2,b_2\}\}$ and $G-N[\{v_2,a_1\}]= \emptyset$, $G-N[\{v_2,b_1\}] \cong P_4 \cong G-N[\{v_2,b_2\}]$. Thus, using \Cref{thm:main theorem for support} and \Cref{prop:higher_ind_path}, we get the following.
	\begin{equation*}
	\begin{split}
	\ind_2(G) & \simeq \Sigma^2 (\ind_2(G-N[\{v_2,a_1\}])) \vee \Sigma^2 (\ind_2(G-N[\{v_2,b_1\}])) \vee \\ & \hspace*{0.7cm}\Sigma^2 (\ind_2(G-N[\{v_2,b_2\}])) \\
	& \simeq \Sigma^2 (\emptyset) \vee \Sigma^2 (\ind_2(P_3)) \vee \Sigma^2 (\ind_2(P_3)) \\
	& \simeq \mathbb{S}^1\vee \mathbb{S}^3\vee \mathbb{S}^3. \\
	\end{split}
	\end{equation*}
	\begin{figure}[H]
		\centering
		\begin{tikzpicture}
		[scale=1.00, vertices/.style={draw, fill=black, circle, inner sep=1.5pt}]
		\node[vertices, label=above:{$v_1$}] (v1) at (0,0)  {};
		\node[vertices, label=above:{$v_2$}] (v2) at (2,0)  {};
		\node[vertices, label=below:{$b_1$}] (b1) at (5,0)  {};
		\node[vertices, label=above:{$a_1$}] (a1) at (4.5,3)  {};
		\node[vertices, label=below:{$b_2$}] (b2) at (3.5,-3)  {};
		
		%\node[vertices, label=above:{$a_1^1$}] (a11) at (8,3)  {};
		\node[vertices, label=above:{$b_1^1$}] (b11) at (7,0)  {};
		\node[vertices, label=above:{$b_1^2$}] (b12) at (9,0)  {};
		\node[vertices, label=above:{$b_1^3$}] (b13) at (11,0)  {};
		\node[vertices, label=above:{$b_1^4$}] (b14) at (13,0)  {};
		%\node[vertices, label=above:{$b_1^5$}] (b15) at (15,0)  {};
		
		\node[vertices, label=below:{$b_2^1$}] (b21) at (7,-3)  {};
		\node[vertices, label=below:{$b_2^2$}] (b22) at (8.5,-3)  {};
		\node[vertices, label=below:{$b_2^3$}] (b23) at (11,-3)  {};
		\node[vertices, label=below:{$b_2^4$}] (b24) at (12.5,-3)  {};
		%\node[vertices, label=below:{$b_2^5$}] (b25) at (15,-3)  {};

		%\node[vertices,inner sep=0.3pt] (d1) at (-0.5,0.8)  {};
		%\node[vertices,inner sep=0.3pt] (d2) at (0.3,0.7)  {};
		%\node[vertices,inner sep=0.3pt] (d3) at (1.1,0.65)  {};
		
		\foreach \to/\from in {v1/v2, v2/b1, v2/a1, v2/b2, b1/b2, a1/b2, a1/b1, b1/b11, b11/b12, b12/b13, b13/b14, b2/b21, b21/b23, b23/b24, a1/b11, a1/b12, a1/b13, a1/b14, a1/b21, a1/b22, a1/b23, a1/b24,  b2/b11, b2/b12, b2/b13, b2/b14, b1/b21,b1/b22, b1/b23, b1/b24}
		\draw [-] (\to)--(\from);
		\end{tikzpicture}\caption{}\label{fig:reverse construction}
	\end{figure}
	
	\begin{proof}[Proof of \Cref{thm:reverseconstructioninside}] Let $P_{r}$ be a path graph on vertex set $\{v_1, v_2, \ldots, v_{r}\}$. For each $i \in \{1,\dots,n\}$, let $W_i$ be a set of cardinality $d_i$. Let $W= \bigsqcup \limits_{i=1}^n W_i$ and $K_{W+1}$ be the complete graph on vertex set $W \sqcup \{v_{r}\}$. For each $x \in W_i$, let $P_x^i$ be a path graph   on $(r+2)(k_i-1)+1$ vertices with $x$ as an end vertex. Here, $V(P_x^i \cap (P_r \cup K_{W+1})) = \{x\}$ for each $x \in W$ and $V(P_x^i \cap P_y^j)= \emptyset$ whenever $x\neq y$. Let 
		$$
		G = P_{r} \cup K_{W+1} \bigcup\limits_{i=1}^{n} (\bigcup\limits_{x \in W_i} P_x^i).
		$$
		
		Let $\tilde{G}$ be the graph with vertex set $V(G)$ and $E(\tilde{G})= E(G) \sqcup \bigcup\limits_{1 \leq i, j \leq n} \{(a,b) : a \in W_j,~ b \in P_x^{i}, \text{ where either } i \neq j \text{ or }  b\neq x\}$.
		Clearly, $v_1$ is a simplicial vertex in $\tilde{G}$ as $N_{\tilde{G}}(v_1) = \{v_2\}$ and $\supp_r(v_1, \tilde{G}) = \{ \{v_2, \ldots, v_{r}, x\} \ : \  x \in W \}$. For each $x \in W$, let $S_x = \{v_1, \ldots, v_r, x\}$. 
		Then by Theorem \ref{thm:main theorem for support}, 
		$\ind_r(\tilde{G}) \simeq \bigvee\limits_{x \in W} \Sigma^r (\ind_r(\tilde{G} - N[S_x]))$. Observe that for each
		$x \in W_i$,  $\tilde{G}- N[S_{x}]$ is isomorphic to a path on $(r+2)(k_i-1)-1$ vertices  and therefore $\ind_r(\tilde{G}-N[S_{x}])\simeq \mathbb{S}^{r(k_i-1)-1}$ by \Cref{prop:higher_ind_path}. Hence, $\ind_r(\tilde{G}) \simeq \bigvee\limits_{\substack{1 \leq i \leq n \\ x \in W_i}} \Sigma^r (\mathbb{S}^{r(k_i-1)-1})= \bigvee\limits_{\substack{1 \leq i \leq n \\ x \in W_i}} \mathbb{S}^{rk_i-1} = \bigvee\limits_{i=1}^n \vee_{d_i} \mathbb{S}^{rk_i-1}$.
	\end{proof}
	
	\section*{Acknowledgements}
PD and AS are partially supported by a grant from Infosys Foundation. PD is also supported by the
MATRICS grant MTR/2017/000239.

%=======================================================
	
\bibliographystyle{plain}

\end{document}